\documentclass[12pt]{amsart}
\usepackage{amsmath, amscd,amsthm, amsfonts}

\newtheorem{theorem}{Theorem}[section]
\newtheorem{lemma}[theorem]{Lemma}
\theoremstyle{definition}

\newtheorem{proposition}[theorem]{Proposition}
\theoremstyle{remark}

\numberwithin{equation}{section} \theoremstyle{plain}

\newtheorem{corollary}{Corollary}

\def\D{\mathbb D}
\def\C{\mathbb C}

\def\F{\mathbb F}
\def\H{\mathbb H}
\def\J{\mathbb J}

\def\R{\mathbb R}
\def\U{\mathbb U}
\def\V{\mathbb V}
\def\W{\mathbb W}

\def\m{\mathcal M}
\def \o{\ddot{\hbox{o}}}
\def \x{{\bf x}}

\begin{document}

\copyrightinfo{2001}{enter name of copyright holder}
\newcommand{\secref}[1]{section~\ref{#1}}
\newcommand{\thmref}[1]{Theorem~\ref{#1}}
\newcommand{\lemref}[1]{Lemma~\ref{#1}}
\newcommand{\rmkref}[1]{Remark~\ref{#1}}
\newcommand{\Propref}[1]{Proposition~\ref{#1}}
\newcommand{\Corref}[1]{Corollary~\ref{#1}}
\newcommand{\eqnref}[1]{~{\textrm(\ref{#1})}}

\title
{$z$-Classes of Isometries of the Hyperbolic Space \footnote{ \it to appear in Conf. Geom. Dyn.}}
\author
{Krishnendu Gongopadhyay \and Ravi S. Kulkarni}
\address
{Indian Institute of Technology (Bombay), Powai, Mumbai 400076, India}

\address{{\it Current Address:} School of Mathematics, Tata Institute of Fundamental Research, Colaba, Mumbai 400005, India.}

\address{Indian Institute of Technology (Bombay), Powai, Mumbai 400076, India,
\and Queens College and Graduate Center, City University of New York.}
\email {krishnendug@gmail.com}
\email{punekulk@yahoo.com}
\date{January 16, 2009}
\footnote{{\it Mathematics Subject Classification(2000). \hspace{.1in}}{Primary 51M10; Secondary 51F25}}


\begin{abstract} Let $G$ be a group. Two elements $x, y$ are said to be {\it $z$-equivalent} if their centralizers are conjugate in $G$.
The class equation of $G$ is the partition of $G$ into conjugacy classes. Further decomposition of  conjugacy classes into
$z$-classes provides an important information about the internal structure of the group, cf. \cite{K} for the elaboration of this theme.

Let $I(\H^n)$ denote the group of isometries of the hyperbolic $n$-space, and let $I_o(\H^n)$ be the identity component of $I(\H^n)$.  We show that the number of $z$-classes in $I(\H^n)$ is finite.
 We actually compute their number, cf.  theorem 1.3. We interpret the finiteness of $z$-classes as accounting for the finiteness of  ``dynamical types" in $I(\H^n)$.
 Along the way we also parametrize conjugacy classes.
We mainly use the linear model of the hyperbolic space for this purpose. This description of parametrizing conjugacy classes appears to be new,
 cf. \cite{CG}, \cite{K2} for previous attempts. Ahlfors \cite{ahlfors} suggested the use of Clifford algebras
to deal with higher dimensional hyperbolic geometry, cf \cite{ahlfors2}, \cite{cw}, \cite{wada}, \cite{waterman}.  These works may be
compared to the approach suggested in this paper.

In dimensions $2$ and $3$, by remarkable Lie-theoretic isomorphisms, $I_o(\H^2)$ and  $I_o(\H^3)$ can be lifted to $GL_o(2, \R)$, and $GL(2, \C)$ respectively. For orientation-reversing isometries there are some modifications of these liftings. Using these liftings,
in the appendix A, we have introduced a single numerical invariant $c(A)$,  to classify the elements of $I(\H^2)$ and  $I(\H^3)$,  and explained the classical terminology.

Using the ``Iwasawa decomposition" of $I_o(\H^n)$, it is possible to equip $\H^n$ with a group structure. In the appendix B, we visualize the stratification of the group $\H^n$ into its conjugacy and $z$-classes.
\end{abstract}

\maketitle

\section{Introduction}
Let $\H^n$ denote the $n$-dimensional hyperbolic space, and $I(\H^n)$ its full isometry group. In the disk-model of the hyperbolic space, an isometry is said
to be {\it elliptic} if it has a fixed point in the space. A non-elliptic isometry is called {\it hyperbolic,} resp. { \it parabolic} if it has $2$, resp.
 $1$ fixed points on the ideal conformal boundary of the hyperbolic space.  We shall build on this classification, and obtain a finer classification up to conjugacy,
 and beyond. A remarkable feature of $\H^n$  is that the group $I(\H^n)$ is infinite. But the ``dynamical types" that our minds can perceive are just finite in number.
 Can we account for this fact in terms of the internal structure of the group alone?

 Let   $I(\H^n)$ act on itself by conjugacy. For $x$ in $I(\H^n)$, let $Z(x)$ denote the centralizer of $x$ in $I(\H^n).$
For $x, y$ in $I(\H^n)$ we shall say $x \sim y$ if $Z(x)$ and $Z(y)$ are conjugate in
$I(\H^n)$. We call the equivalence class defined by this relation, the
{\emph{$z$-class}} of $x$. It turns out that the number of $z$-classes is finite, and
this fact is interpreted as accounting for  the finiteness of dynamical types. The $z$-classes are pairwise disjoint, and are manifolds,
cf. theorem 2.1 \cite{K}, and so give a stratification of $I(\H^n)$ into finitely many strata.   Each stratum has two canonical fibrations, and
they explain the ``spatial" and ``numerical" invariants which characterize the transformation. The $I(\H^n)$ provides a significant example of the
philosophy that was suggested in \cite{K}.

In this paper we  work mostly in the linear model. Let $\V$ be a real vector space of dimension $n+1$ equipped with a quadratic form of signature $(1, n)$, and $O(Q)$
its full group of isometries. This group has four components. The hyperboloid
$\{v\in \V\;|\; Q(v) = 1\}$ has two components. One of the components is the model for $\H^n$. Suggestively, let $I(Q)$ denote the group  which preserves the component, and
we identify it with $I(\H^n)$. Now $O(Q)$ is a linear algebraic group. A significant property of a linear algebraic group over characteristic zero is that each  $T$  in a linear algebraic group has the  Jordan decomposition $T = T_sT_u$ where $T_s$ is semisimple, (that is, every   $T_s$-invariant subspace has a $T_s$-invariant complement), and $T_u$
is unipotent, (that is, all eigenvalues of $T_u$ are $1$) cf. \cite{hum}. Then $T_s$ and $T_u$ are also in $O(Q)$. $T_s$ and $T_u$ commute. They are polynomials in $T$. The decomposition is unique.
  A systematic use of Jordan decomposition leads to a neat and more refined classification of elements of $I(Q)$ up to conjugacy, and as we shall see, also $z$-equivalence.

Let $T$ be in $O(Q)$. Let $\V_c = \V \otimes_\R \C$  be its complexification. We identify $T$ with $T\otimes_{\R} id$ and also consider it as an operator on $\V_c$.
We say that an
eigenvalue $\lambda$ of $T$ is {\it pure} if the corresponding eigenspace $\{v \in \V| (T- \lambda)v = 0\}$, coincides with the generalised
eigenspace $\{v \in \V| (T-\lambda)^{n+1} v = 0\}$. Otherwise $\lambda$ is {\it mixed}.

It is customary to call $v\in V$ {\it time-like} (resp. {\it space-like}, resp. {\it light-like}) if
$Q(v) > 0$, (resp. $Q(v) <  0$, resp. $Q(v) = 0$). A subspace $\W$ is  {\it time-like} (resp. {\it space-like}, resp. {\it light-like}) if $Q|_{\W} > 0$
(resp. $Q|_{\W} <  0$, resp. $Q|_{\W} = 0$).

The roots of a polynomial of the form $x^2 - 2ax + 1$, with $|a| < 1$ are of the form
$e^{\pm i\theta}$. With the restriction $0 < \theta < \pi$, the $\theta$ is uniquely determined. When $x^2 - 2ax + 1$ is a factor of the characteristic
polynomial of a transformation, the $\theta$ will be called {\it the rotation angle} of the transformation.

\begin{theorem} \label{main}{Let $T$ be in $I(Q)$.

i) Only the eigenvalue $1$ can be mixed.

ii) The eigenvectors corresponding to eigenvalue $-1$ are necessarily space-like. $T$ is orientation-preserving iff the multiplicity of the
eigenvalue $-1$ is even. Equivalently,
$T$ is orientation-preserving iff $det \, T = 1$.

iii) Suppose $T$ has a time-like eigenvector. Then $T$ is semisimple, and its characteristic polynomial $\chi_T(x)$ has the form

 $\chi_T(x) = (x-1)^l(x+1)^m \chi_{oT}(x)$, where $\chi_{oT}(x) = \Pi_{j=1}^{s}
 (x^2 - 2a_jx + 1)^{r_j}$, $|a_j| < 1.$

\noindent Here $a_j$'s are distinct, and $l\ge 1$. Let $k = \sum_{j=1}^{s}r_j + [\frac{m}{2}]$. We agree to call $T$ a {$k$-rotatory elliptic} if $m$ is even,
and a {$k$-rotatory elliptic inversion} if $m$ is odd.

iv) $T$ can have only one pair of real eigenvalues $(r, r^{-1})$, $r \not= \pm 1$. Such $r$ is necessarily positive. The eigenspaces corresponding to $r, r^{-1}$
are $1$-dimensional, and are light-like.
Suppose this is the case. Then $T$ is semisimple, and its characteristic polynomial $\chi_T(x)$ has the form

 $\chi_T(x) = (x-1)^l(x+1)^m \chi_{oT}(x)$, where $\chi_{oT}(x) =  (x-r)(x-r^{-1})\Pi_{j=1}^{s}(x^2 - 2a_jx + 1)^{r_j}$, $|a_j| < 1.$ Here $a_j$'s are distinct.
Let $k = \sum_{j=1}^{s}r_j + [\frac{m}{2}]$. We agree to call $T$ a {$k$-rotatory hyperbolic} if $m$ is even, and a {$k$-rotatory hyperbolic inversion} if $m$ is odd.

 v) Suppose $T$ is not semisimple. Then $T$ has a $1$-dimensional    light-like eigenspace with eigenvalue $1$. Let $T = T_sT_u$ be its Jordan decomposition,
 $T_u \not= Id$. The characteristic polynomial $\chi_T(x)$ has the form

 $\chi_T(x) = (x-1)^l(x+1)^m \chi_{oT}(x)$, where $\chi_{oT}(x) = \Pi_{j=1}^{s}(x^2 - 2a_jx + 1)^{r_j}$, $|a_j| < 1.$

\noindent Here $a_j$'s are distinct, and $l \ge 3$. Let $k = \sum_{j=1}^{s}r_j + [\frac{m}{2}]$. We agree to call $T$ a {$k$-rotatory parabolic} if $m$ is even,
and a {$k$-rotatory parabolic inversion} if $m$ is odd. $T_s$ is elliptic, and the minimal polynomial of $T_u$ is $(x-1)^3.$

vi) Every $T\in I(Q)$ satisfies either iii), or iv), or v).}
\end{theorem}

\vskip .3in

A significant consequence of theorem 1.1 is the classification of conjugacy classes.

\vskip .3in

\begin{theorem}\label{conjcl}
{The characteristic polynomial and the minimal polynomial
of $T$ in $I(Q)$ determine the conjugacy class of $T$.}
\end{theorem}

We call $T$  {\it elliptic} if it is {$k$-rotatory elliptic} or {$k$-rotatory elliptic inversion} for some $k$. Similarly for $hyperbolic$ or $parabolic$.
It will be convenient to call
$\chi_{o T}(x)$, the {\it reduced characteristic polynomial} of $T$. Its degree is always even. Let $k' = \frac{1}{2}\; deg \,  \chi_{o T}(x)$.

The number of conjugacy classes is infinite. This infinity arises roughly from the eigenvalues  $r$ and the rotation-angles $\theta_j$ with multiplicity $r_j$ in
the notation of theorem 1.1. These are the ``numerical invariants" of a transformation. The ``spatial invariants" of a transformation are the corresponding
orthogonal decomposition of the space, and the signatures of the quadratic form of the summands. Roughly speaking the ``spatial invariants" define the $z$-class.

We determine the centralizers of elements and $z$-classes in \S 4, and also count their number.

\begin{theorem}\label{count}{For $u$ a natural number, let $p(u)$ denote the Eulerian partition function, $p(0) = 1$.

i) There are $(n + 1) p(0) +  (n -1) p(1) + (n - 3) p(2) +  \ldots +  \epsilon p([\frac{n}{2}])$
where $\epsilon = 1$ or $2$ according as $n$ is even or odd,  \noindent $z$-classes of elliptic elements. For this count identity is considered as an elliptic element.

ii) There are  $([\frac{n-1}{2}] + 1) p(0) +  ([\frac{n-1}{2}])  p(1) +  ([\frac{n-1}{2}] - 1)  p(2)  \ldots +   p([\frac{n-1}{2}])$
\noindent $z$-classes of hyperbolic  elements.

iii) There are $(n-1) p(0) +  (n-3)  p(1) +    \ldots +  \epsilon p([\frac{n-2}{2}])$
where $\epsilon = 1$ or $2$ according as $n$ is even or odd,
\noindent $z$-classes of parabolic  elements.
}
\end{theorem}

Each $z$-class is a manifold, and so  each component has a well-defined dimension. It is said to be {\it generic} if it has
the maximum dimension ( $= dim\; I(Q)$).

\begin{theorem}\label{generic}
{ i) Let $n$ be even, $n = 2n'$. The generic $z$-classes are the $z$-classes
of

\noindent $n'$-rotatory elliptics,    $(n'-1)$-rotatory hyperbolics,
and  $(n'-1)$-rotatory hyperbolic

\noindent inversions, all with pairwise distinct rotation angles.

ii) Let $n$ be odd, $n = 2n' + 1$. The generic $z$-classes are the $z$-classes of $n'$-rotatory   hyperbolics, $n'$-rotatory elliptic inversions, and $(n'-1)$-rotatory
hyperbolic inversions, all with pairwise distinct rotation angles.
}

\end{theorem}

\vskip .3in

In \S 5 we give some simple criteria to detect the type of  an isometry  of $I(\H^n)$ to be

\noindent hyperbolic, resp. elliptic, resp. parabolic. For example, $T$ is hyperbolic iff $\chi_{oT}(1) <  0$. $T$ is parbolic, resp. elliptic,  iff $\chi_{oT}(1) >  0$ and $(x-1)^2$
divides (resp. does not divide) the minimal polynomial of $T$. There are also criteria in terms of
the trace of $T$.

Previously Chen and Greenberg \cite{CG}, Kulkarni  \cite{K2}  analysed the conjugacy classes of hyperbolic isometries. The criterion for conjugacy provided in Theorem 1.2
appears to be  new.   Ahlfors, \cite{ahlfors}, suggested the use
of Clifford algebras to study M\"obius transformations in higher dimensions, cf. \cite{ahlfors2}, \cite{cw}, \cite{wada}, \cite{waterman}
 for elaboration of this theme. The present paper presents an alternate  approach. It appears that Clifford-algebraic approach has not yet been worked out for
 orientation-reversing transformations.

\section{Proof of theorem 1.1}

1) Let $\V_c$ denote the complexification of $\V$, and $Q_c$ the bilinear extension of $Q$ to $\V_c$. Let $T$ be a linear transformation of $\V_c$, and
its Jordan decomposition $T = T_sT_u$. Recall some basic facts about this set up. Let
$\V_{c, \lambda}$ denote the generalized eigenspace of $T$ with eigenvalue $\lambda$. Then it is the (usual) eigenspace of $T_s$.  Let $T$ be in $O(Q_c)$, and
let $< , >$ denote the corresponding bilinear form. We have for $v, w \in \V_{c, \lambda}$

$<v, w>  =  <Tv, Tw>  =  <T_sv, T_sw>  =  \lambda^2 <v, w>.$

\noindent So if $<v, w>  \not= 0$, then $\lambda = \pm 1$. Or to put it another way, if $\lambda \not= \pm 1$ then $Q|_{\V_{c, \lambda}} = 0.$

\noindent Also for $v \in \V_{c, \lambda}$ and $ w \in \V_{c, \mu}$ we have

$<v, w>  =  <Tv, Tw>  =  <T_sv, T_sw>  =  \lambda \mu <v, w>.$

\noindent So unless $\lambda \mu =1$ we have $\V_{c, \lambda}$ and $\V_{c, \mu}$ are othogonal with respect to $Q_c$. Let $\oplus$ denote the orthogonal direct sum,
and $+$ the usual direct sum of subspaces. We have

$\V_c = \V_{c, 1} \oplus \V_{c, -1} \bigoplus \oplus_{\lambda \not= \pm 1}(\V_{c, \lambda} +
\V_{c, \lambda^{-1}}).$

\noindent Note in particular, $dim \; \V_{c, \lambda} = dim \; \V_{c, \lambda^{-1}}.$

2) Now let $T$ be in $I(Q).$ Note that a unipotent transformation must preserve

\noindent $\{v \in\V|  Q(v) \ge 0\},$ so $T_u$ and hence $T_s$, are also in $I(Q)$.

Suppose that $T$ has a time-like eigenvector $v$. Since $<v, v>  \not= 0$, we see that the corresponding eigenvalue is equal to $1$.   $T$ leaves the orthogonal complement
of $v$ invariant, and the metric is negative definite on this complement. So $T$ must be semisimple.  The characteristic polynomial $\chi_T(x)$ is of the form

$\chi_T(x) = (x-1)^l(x+1)^m \chi_{oT}(x)$, where $\chi_{oT}(x) = \Pi_{j=1}^{s}(x^2 - 2a_jx + 1)^{r_j}$, $|a_j| < 1.$

\noindent Here $a_j$'s are distinct. We have $l \ge 1$. The minimal polynomial $m_T(x)$ divides the characteristic polynomial, it has the same factors as the characteristic
polynomial, and for a semisimple element it is a product of distinct factors. So

$m_T(x) = (x-1)(x+1)^{\epsilon}\Pi_{j=1}^{s}(x^2 - 2a_jx + 1)$.

\noindent where $\epsilon = 0 $ if $m = 0,$ and $\epsilon = 1 $ if $m > 0$.

\noindent Let $k = \sum_{j=1}^{s}r_j + [\frac{m}{2}]$. We have agreed to call such $T$ a  {$k$-rotatory elliptic} if $m$ is even, and a {$k$-rotatory elliptic inversion}
if $m$ is odd.

 3) Let $T$ be in $I(Q)$ and suppose that $r$ is a real eigenvalue $\not= 1$. By 1) the corresponding eigenvector is light-like.
Now $T$ preserves $\{v \in \V | <v, v> \ge 0\}$.
 So $r$ must be positive. By 1) $r^{-1}$ must also be an eigenvalue, since $Q$ is non-degenerate. The dimension of a subspace on which $Q = 0$ is at most $1$. So
$dim \; \V_r =  dim \: \V_{r^{-1}} = 1$, and $Q|_{\V_r + \V_{r^{-1}}}$ is non-degenerate.
Now  the orthogonal complement of $\V_r + \V_{r^{-1}}$ is $T$-invariant, and $Q|_{\V_r + \V_{r{-1}}}$ is negative semidefinite. So $T$ is semisimple. The characteristic
polynomial $\chi_T(x)$ is of the form

$\chi_T(x) = (x-1)^l(x+1)^m \chi_{oT}(x)$, where $\chi_{oT}(x) = (x-r) (x - r^{-1})\Pi_{j=1}^{s}(x^2 - 2a_jx + 1)^{r_j}$.

\noindent Here $|a_j| < 1$, and  are distinct. Moreover

 $m_T(x) = (x-1)^{\epsilon_1}(x+1)^{\epsilon_2}(x-r)(x-r^{-1})\Pi_{j=1}^{s}(x^2 - 2a_jx + 1)$,

\noindent where $\epsilon_1$ (resp. $ \epsilon_2$) is $0$ if $l$ (resp. $m$) is $0$, and
equal to $1$ if $l$ (resp. $m$) $ > 0.$

\noindent Let $k = \sum_{j=1}^{s}r_j + [\frac{m}{2}]$. We have agreed to call such $T$ a  {$k$-rotatory hyperbolic} if $m$ is even, and a {$k$-rotatory hyperbolic
inversion} if $m$ is odd.

4) Now let $T$ be in $I(Q)$, and $\lambda$ a non-real complex eigenvalue. Let $v \in \V_c$ be a corresponding eigenvector. Write $v = v_1 + i v_2$ where $v_1$ resp. $v_2$
are real, and $\lambda = a + ib$, where $a, b$ are real and $b \not= 0$. We have $Tv = \lambda v$, so

$Tv_1 = av_1 - bv_2, Tv_2 = bv_1 + a v_2.$

\noindent Since $b \not=0$, we see that both $v_1, v_2$ are non-zero, and they are linearly independent. Let $\W = span \{v_1, v_2\}$. It is a $2$-dimensional $T$-invariant
subspace. It is indeed a {\it minimal} invariant subspace. So $Q|_{\W}$ must be either $0$, or it must be non-degenerate, for otherwise the nullspace of $\W$
would be $1$-dimensional and would be $T$-invariant. Also $Q|_{\W}$ cannot have signature $(1, 1)$,
for an elementary fact from $2$-dimensional Lorentzian geometry is that
the eigenvalues of $T|_{\W}$ would be real. So $Q|_{\W}$ must be negative definite.
But then the eigenvalues of $T$ on $\W\otimes_{\R}\C$ are $(\lambda, \bar{\lambda})$, and must be of the form $(e^{i\theta}, e^{-i\theta})$.
By interchanging $\lambda$ with $\bar{\lambda}$ if necessary, we may take  $0 < \theta < \pi$. By making this argument on $\V_{c, \lambda}$
we see that $\V_{c, \lambda} + \V_{c, \bar {\lambda}}$ is a complexification of a real subspace on which $Q$ is negative definite. Let us denote this real
subspace  $\V_{\lambda, \bar {\lambda}}$, and call it  the {\it real trace} of the eigenspace of $\V_{c, \lambda} + \V_{c, \bar {\lambda}}$.

5) Suppose $-1$ is an eigenvalue of $T$ and $v$ a corresponding eigenvector. Since $T$ preserves the set
$\{u \in \V| Q(u) \ge 0\}$, $v$ must be outside this set. In other words, $v$ is space-like. Now combining 3), 4) we see that $1$
is the only possible eigenvalue which is mixed.

6) $\V$ is said to be {\it orthogonally indecomposable} with respect to $T$, or $T$-{\it orthogonally indecomposable} for short,  if $\V$ is not an orthogonal
 direct sum of proper $T$-invariant subspaces. Given $T$, we can decompose $\V$ into $T$-orthogonally indecomposable subspaces $\W$, $Q|_{\W}$ is non-degenerate.
So far we know three types of $T$-orthogonally indecomposable subspaces.

i) Dim $\W$ = 1: Here $\W$ is spanned by an eigenvector of non-zero length.

ii) Dim $\W$ = 2: Here $T|_{\W}$ has no real eigenvalue and $Q|_{\W}$ is negative definite. The eigenvalues of $T$ are $e^{\pm i\theta}$, where $0 < \theta < \pi$.

iii) Dim $\W = 2$: Here $\W = \V_r + \V_{r^{-1}}$, where both $\V_r$ and $\V_{r^{-1}}$
are spanned by eigenvectors of length $0$.

We now describe the fourth, and last, type of $T$-orthogonally indecomposable subspaces.

\vskip .3in

\begin{lemma}{Let $\W$ be an $T$-orthogonally indecomposable subspace of dimension $\ge 3$. Then $T$ is unipotent and $dim \; \W = 3$.}
\end{lemma}

\begin{proof} Let $dim \; \W = m$. From previous discussion we see that $T|_{\W}$ must be unipotent.    Let $v$ be an eigenvector. From 3) and 4) we must have
$v$ of length $0$, and eigenvalue $1$. Let $C = \{u \in \W| Q(u) = 0\}.$ This is a right circular cone. Since $Tv = v,$
$T$ must preserve the tangent space $\tau$ at $v$ to $C.$ Now $dim \; \tau = m - 1,$
and $Q|_{\tau}$ is degenerate. $Q|_{\tau}$ has nullity $1$, spanned by $v$,  and otherwise its signature is $(0, m-2)$. Let $\tau \rightarrow \tau/<v>$ be the
canonical projection. $Q$ induces the canonical symmetric quadratic form $\bar {Q}$ on
$\tau/<v>$. $T$ induces a linear transformation $\bar {T}$ on $\tau/<v>$, and it preserves $\bar {Q}$. Now $\bar {Q}$ has signature $(0, m-2)$,
so it is negative definite. $T$ is unipotent, so is $\bar T$.
So $\bar T$ is both semisimple and unipotent, so it must be $ I $, the identity. It is now easy to see that since we assumed $\W$ is
orthogonally indecomposable with respect to $T$, we must have $dim \; \tau/<v>   =  m - 2  = 1.$ So $dim \; \W = 3.$

\end{proof}

7)   Let now $T$ be a non-semisimple element of $I(Q)$. Let $T = T_sT_u$, $T_u \not= I$ be the Jordan decomposition of $T$.

Let $\lambda$ be a non-real eigenvalue of $T_s$. In 4) we have
seen that $\V_{c, \lambda} + \V_{c, \bar {\lambda}}$ is a
complexification of its trace $\V_{\lambda, \bar {\lambda}}$ on
which $Q$ is negative definite.  Now $T_u$ commutes with $T_s$. So
$T_u$  keeps the eigenspace of $T_s$ in $\V_c$ invariant. It
follows that $T_u$ leaves $\V_{c, \lambda} + \V_{c, \bar
{\lambda}}$, and hence $\V_{\lambda, \bar {\lambda}}$ invariant.
Now $Q$ restricted to ${\V_{\lambda, \bar {\lambda}}}$ is negative
definite. So $O(\V_{\lambda, \bar {\lambda}})$ is compact. A
compact group does not contain a   non-semisimple element. It
follows that ${T_u}$ restricted to ${\V_{\lambda, \bar
{\lambda}}}$ is identity.

Let $-1$ be an eigenvalue of $T_s$. By 5) the corresponding eigenvector is space-like.
It follows that $Q$ restricted to $\V_{-1}$ is negative definite. By the above argument
$T_u$ restricted to $\V_{-1}$ is identity also.

Let $r$ be a (necessarily positive) real eigenvalue $\not= 1$ of $T_s$. By 4), $T$ is hyperbolic. So $T = T_s$ and $T_u = identity$. By our hypothesis, this is not the case.
So this case cannot occur.

So $T_u$ can be different from identity only on the eigenspace $\V_1$ with eigenvalue $1$ of $T_s$. But $T_s$ being semisimple, it must be itself identity on   $\V_1$.

We agree to call such $T$ {\it parabolic}. The characteristic polynomial of such $T$ has the form

$\chi_T(x) = (x-1)^l(x+1)^m \chi_{oT}(x)$, where $\chi_{oT}(x) = \Pi_{j=1}^{s}(x^2 - 2a_jx + 1)^{r_j}$, $|a_j| < 1.$

\noindent Here $a_j$s are distinct. We have $l \ge 3$. The minimal polynomial has the form

$m_T(x) = (x-1)^{l'}(x+1)^{\epsilon}\Pi_{j=1}^{s}(x^2 - 2a_jx + 1)$.

\noindent Here $l' = 3$. Moreover $\epsilon = 0$ if $m = 0$, and $\epsilon = 1$ if $m > 0$.  Let $k = \sum_{j=1}^{s}r_j + [\frac{m}{2}]$.  In a more refined way, we call $T$
  a  {$k$-rotatory parabolic} if $m$ is even, and a {$k$-rotatory parabolic inversion} if $m$ is odd. Notice also that $m_{T_u}(x) = (x-1)^3$.

  This completes the proof.

  \section{Conjugacy Classes}

 If the transformations are conjugate their characteristic and minimal polynomials are the same, as indeed this is true for elements of $GL(\V)$.
The interesting point is the converse.

 Suppose $T$ is elliptic. Then $\V$ is an orthogonal direct sum of $1$- or $2$-dimensional indecomposable $T$-invariant subspaces. The metric on all $2$-dimensional subspaces
 is negative definite. On one $1$-dimensional subspace it is  positive definite, and on the others it is negative definite. Given $T'$ conjugate to $T$, $\V$ is
a similar orthogonal direct sum of $1$- or $2$-dimensional indecomposable $T'$-invariant subspaces. There is an isometry in $O(Q)$ carrying one decomposition into another.
The isometry may be chosen to lie  in $I(Q)$ for both $T, T'$ preserve $\{v \in \V| Q(v) \ge 0\}.$ This isometry  conjugates $T$ into $T'$.

 The proof in the hyperbolic case is similar. In this case $\V$ is an  orthogonal direct sum of one $2$-dimensional subspace on which $Q$ has signature $(1, 1)$.
The rest is similar to the elliptic case.

 Suppose $T$ is parabolic. Then $\V$ is an orthogonal direct sum of one $3$-dimensional

\noindent $T$-indecomposable subspace on which the metric has the signature $(1, 2)$, and the other $1$ or $2$-dimensional $T$-invariant subspaces on which the metric is
negative definite. Now the proof is similar to the elliptic case.

Notice that if we know a transformation is not parabolic then the  characteristic polynomial itself determines the conjugacy class.

\section{Centralizers, and $z$-classes}

1) Let $T$  be an elliptic  transformation. First let us note some invariants of $T$ from a dynamic viewpoint. Let $\Lambda$ be the set of fixed points of $T$.
Then $\Lambda$ is the eigenspace of $T$ with eigenvalue $1$. Now $Q|_{\Lambda}$ is non-degenerate, cf. part 1) of \S 2. Let $\Pi$ be the subspace orthogonal
to $\Lambda$. Then  $\V = \Lambda \oplus \Pi$. The centralizer $Z(T)$ in $I(Q)$ leaves each  eigenspace, or a trace of a complex eigenspace, of $T$ invariant.
So it is clear that $Z(T|_{\Lambda}) =
 I(Q|_{\Lambda})$.

Note that $Q|_{\Pi}$ is negative definite. In $\Pi$ let $\Pi_{-1}$ be the eigenspace with
 eigenvalue $-1$, and $\Pi_o$ be its orthogonal complement. Let

 $\chi_{T|_{\Pi_o}}(x) = \Pi_{j=1}^{s}(x^2 - 2a_jx + 1)^{r_j}$, $|a_j| < 1.$

 \noindent Here $a_j$'s are distinct. Let $\Pi_j = ker \; (T^2 - a_jT + 1)^{r_j}$. Then
 $\Pi = \oplus_{j=1}^{s} \Pi_j$, and

\noindent $dim \; \Pi_j = 2r_j$.

 Let $v$ be a vector in $\Pi_j$, $<v, v>  =  -1$,   and $Tv = v_1$.
 Then $span\; \{v, v_1\} $
 is $T$-invariant. Choose orientation on  $span \; \{v, v_1\} $ so that the angle between
 $v$ and $v_1$ is $\theta$, where $0 < \theta < \pi$. Let $w$ be a vector in $span\; \{v, v_1\} $, $<w, w> \, =  -1$, such that $(v, w)$ is a positively
oriented orthogonal pair in
 $span\; \{v, v_1\} $. Let $\J_j(v) = w, \J_j(w) = -v$. Then $\J_j$ is a $T$-invariant complex structure on $span\; \{v, v_1\} $. We easily see that $\J_j$, in
 $span\; \{v, v_1\} $, is independent of the choice of $v$. Thus this process in fact defines $\J_j$ as a $T$-invariant complex structure on $\Pi_j$.
It follows that the $Z(T|_{\Pi_j})$  is isomorphic to the unitary group $U(r_j)$.

Now  $Z(T|_{\Pi_{-1}})$ is clearly $O(Q|_{\Pi_{-1}})$, and it is compact.

So $Z(T)$ is isomorphic to $I(Q|_ {\Lambda}) \times O(Q|_{\Pi_{-1}}) \times \Pi_{j=1}^{s}U(r_j)$.

\vskip .3in
2) Let $T$ be a hyperbolic transformation. Here we have a unique
$2$-dimensional

\noindent nondegenerate $T$-invariant subspace $\W$, which is the sum of two eigenspaces with

\noindent eigenvalues $\not=\pm 1$, and  on which the metric is of the form $(1, 1).$
$Z(T)$ leaves $\W$ invariant. It is easy to see that $Z(T|_{\W}) =  I_o(Q|_{\W}),$ where $I_o(Q)$ is the identity component of $ I(Q)$. On the orthogonal complement of $\W$,
the metric is negative definite. Now the analysis is similar to the   elliptic case. We have $Z(T)$  isomorphic to

$I_o(Q|_{\W}) \times O(Q|_ {\Lambda}) \times O(Q|_{\Pi_{-1}}) \times \Pi_{j=1}^{s}U(r_j)$.

\noindent Notice that in this case $ O(Q|_ {\Lambda}) $ is necessarily compact.
\vskip .3in

3) Now let $T$ be a parabolic transformation, $T = T_sT_u$. We know that $T_s$ and $T_u$ are polynomials in $T$. So if $A$ commutes with $T$, it commutes with $T_s$
and $T_u$ as well. Conversely if $A$ commutes with $T_s$ and $T_u$ it commutes with
$T$. In other words we have $Z(T) = Z(T_s) \cap Z(T_u).$

We have a decomposition $\V = \W \oplus \U$, where $T_u|_{\U} = identity$,   $dim \; \W = 3$, and $Q|_{\W}$ has signature $(1, 2)$. Also $T_s|_{\W} = identity$.
On $\U$ the metric is negative definite. Decompose $\U$ into $\U_1 \oplus \U_{-1} \oplus \U_2$, where $\U_1$ (resp. $\U_{-1}$) are the eigenspaces with eigenvalue $1$
(resp. $-1$), and $\U_2$ is the orthogonal sum of the real traces of the complex eigen-spaces. Since

\noindent $T_u|_{\U} = identity,$
we have $Z(T|_{\U_{-1} \oplus \U_2}) = Z({T_s}|_{\U_{-1} \oplus \U_2})$ is a direct product of $O(\U_{-1})$ and various unitary groups. This leaves
 $Z(T|_{\W \oplus \U_1}) = Z({T_u}|_{\W \oplus \U_1})$. To compute $Z({T_u}|_{\W \oplus \U_1}) $ it is convenient to use the upper half space model of the hyperbolic
space, where $\infty$ is the fixed point of $T_u$. Let $dim \; \W \oplus \U_1 = n' + 1$. Take the upper half space, which we conveniently again denote by $\H^{n'}$,
with orthogonal coordinates $(x_1, x_2, \ldots , x_{n'})$, with $x_{n'} \ge 0$.
 The $(n'-1)$-planes   $x_{n'} =  c > 0$, are parts of the horo-spheres. We may further assume that the coordinates are so chosen that  $T_u$ is given by
 $(x_1, x_2, \ldots , x_{n'})  \mapsto (x_1 + 1, x_2, x_3, \ldots , x_{n'})$. Now the orthogonal group in coordinates $(x_2, x_3, \ldots, x_{n' - 1})$
commutes with $T_u$, as also  the translations
 in the planes $x_{n'} = c$. The structure of  the group may be described as follows. It is a direct product of two groups.

 i) a copy  of $\R$: Here $\R$ is the group of translations
 $(x_1, x_2, \ldots , x_{n'})  \mapsto (x_1 + c, x_2, x_3, \ldots , x_{n'})$.

ii) a semidirect product of  $ \R^{n' -2}$ by $O(n' - 2)$: Here $ \R^{n' -2}$ is the group of  translations
$(x_1, x_2, \ldots , x_{n'})  \mapsto (x_1, x_2 + c_2, x_3 + c_3, \ldots , x_{n' - 1} + c_{n' -1}, x_{n'}) $, and $O(n' - 2)$ is the orthogonal group
in coordinates $(x_2, x_3, \ldots, x_{n'-1})$.

\vskip .3in

From this description of centralizers we see that to each centralizer there is attached a certain orthogonal decomposition of the space, and its isomorphism
type is determined by it. Given any two such decompositions of
the same type (meaning that the dimensions and the signatures of the quadratic form for the corresponding summands are the same) there exists an element
of $I(Q)$ carrying one decomposition into the other. This element will conjugate one centralizer into the other. Thus there are only finitely many $z$-classes.
Now the centralizers are described in terms of group-structure itself. This fact is interpreted as accounting for the finiteness of  ``dynamical types".

\vskip .3in

The actual count of $z$-classes requires more work, and involves counting certain types of partitions of $n+ 1$.
Let $T$ be in $I(Q)$, $\chi_T(x) = (x-1)^l(x+1)^m \chi_{oT}(x)$ be its characteristic polynomial, where
$$\chi_{oT}(x) = \Pi_{j=1}^{s}(x^2 - 2a_jx + 1)^{r_j}, $$
$$\hbox{or},\;\;\chi_{oT}(x) = (x-r) (x- r^{-1})\Pi_{j=1}^{s}(x^2 - 2a_jx + 1)^{r_j},$$
$|a_j| <1$, and distinct. We associate to $T$
the partition $n + 1 = l + m + \sum_{j=1}^{s} r_j$, or
  $n + 1 = 2+ l + m + \sum_{j=1}^{s} r_j$. Notice that knowing that $T$ is elliptic, resp. hyperbolic, resp. parabolic, and the partition of $n + 1$
determine its $z$-class uniquely. It is easy to see that distinct partitions (and knowing that $T$ is elliptic, resp. hyperbolic, resp. parabolic)
are associated to distinct $z$-classes.  In considering $z$-classes we consider identity as elliptic. There are some restrictions on which partitions can actually occur.

First observe that in any group the center forms a single $z$-class.

\begin{proposition}{The center of $I(Q)$ is trivial.}
\end{proposition}

\begin{proof}   Let $g$ be an element of the center. In the disk-model we see that $g$ must have a fixed point. The fixed point set of a central element
is left invariant by the whole group. Suppose $g$ has a fixed point in the open disk. Now $I(Q)$ is transitive on the open disk, so $g$ must fix all the points.
In other words $g$ must be identity. Now suppose $g$ has a fixed point on the boundary. But $I(Q)$ is transitive also on the boundary.
So $g$ must fix all the boundary. But then $g$ fixes the convex hull of the boundary. So again $g$ has to be identity.

\end{proof}

{\it Case 1}: Let $T$ be an elliptic element.  Let $n + 1 = l + m + 2\sum_{j=1}^{s} r_j$ be the associated partition.  We have $l \ge 1$ and $ |a_j| < 1$.
Let $k'$ be the number of rotation angles lying in
$(0, \pi)$, $k' = \sum_{j=1}^{s} r_j$. We assume, as we may,  that  the rotation angles are ordered in increasing order. The rotation angle $\theta_j$ is repeated
as many times as its multiplicity $r_j$.   The contribution of this part to the centralizer is the product of unitary groups $U(r_j)$. Thus there are  $p(k')$
choices of distinct rotation angles counting with multiplicities.
Since $l \ge 1$, we have   $ [\frac{n}{2}] + 1$  choices for $k'$ (including 0). For a fixed $k'$,
there are $n + 1  - 2 k'$ choices for $l$. Thus setting $p(0) = 1$,   there are
$(n + 1) p(0) +  (n -1) p(1) + (n - 3) p(2) +  \ldots +  \epsilon p([\frac{n}{2}])$
where $\epsilon = 1$ or $2$ according as $n$ is even or odd,
\noindent $z$-classes of elliptic elements.

 {\it Case 2}: Let $T$ be a  hyperbolic element.  Rewrite the associated partition of $n + 1$ as  $n + 1 = 2 + l + m + 2\sum_{j=1}^{s} r_j$,
where the first $``2"$ stands for a pair of real eigenvalues $\not=\pm 1$, and
 $l$ (resp. $m$) is the multiplicity of the eigenvalue $1$ (resp. $-1$), and $|a_j| < 1$. In this case we see that if $l$ and $m$ are interchanged
we get the same $z$-class. So for counting $z$-classes we may further assume $l \le m$. Thus we have
 $ [\frac{n-1}{2}] + 1$  choices for $k'$ (including 0), and for a fixed $k'$,
there are $[\frac{n - 1 - 2k'}{2}] + 1$ choices for $l$ (including 0).
Thus we have (setting $p(0) = 1$)
 $([\frac{n-1}{2}] + 1) p(0) +  ([\frac{n-1}{2}])  p(1) +  ([\frac{n-1}{2}] - 1)  p(2)  \ldots +   p([\frac{n-1}{2}])$
\noindent $z$-classes of hyperbolic  elements.

{\it Case 3}:  Let $T$ be a  parabolic element. Rewrite the associated partition of $n + 1$ as  $n + 1 = 3 + l + m + 2\sum_{j=1}^{s} r_j$,
where $3 + l$  (resp. $m$) is the multiplicity of the eigenvalue $1$ (resp. $-1$). Now there are $[\frac{n-2}{2}] + 1$ choices for $k'$ (including 0),
and for a fixed $k'$,  we have $n - 1 - 2k'$ choices for $l$ (including $0$).
Thus we have (setting $p(0) = 1$)  $(n-1) p(0) +  (n-3)  p(1) +    \ldots +  \epsilon p([\frac{n-2}{2}])$
where $\epsilon = 1$ or $2$ according as $n$ is even or odd,
 $z$-classes of parabolic  elements.

We have completely proved theorem 1.3.

Lastly note that each $z$-class is a manifold, so one can talk about its dimension for each of its components. A $z$-class may not be connected but there is
actually a single dimension associated to  a $z$-class. It can be read from theorem 2.1 of \cite{K}. From that description we see that in a semisimple
Lie group, the $z$-class is generic, i.e. the dimension of a $z$-class is maximal, iff the centralizer is abelian. The stated $z$-classes in
theorem 1.4 are precisely the  $z$-classes with this property.  This completes
the proof of \thmref{generic}.

 \section{Some simple Criteria}

 In this section we note some simple criteria for detecting the type of   isometries of the
 hyperbolic $n$-space, using the linear model. Let $T$ be in $I(Q)$, and let $\chi_{o T}(x)$ be
 its reduced characteristic polynomial, that is the polynomial obtained after factoring all factors of the type $x\pm 1$.
 Note that the factors $x^2 - 2ax + 1$ for $|a| < 1$
 always takes positive values for real $x$. The factors $(x-r)(x-r^{-1})$ for $r > 1$ are present only in the case of hyperbolic transformation.
The criteria involve only detecting
 the factors $x \pm 1$ of the characteristic or minimal polynomial of $T$. So they are defined over any subfield of $\R$ which contains the
coefficients of the characteristic polynomial of $T$, in particular on any subfield of $\R$ generated by the coefficients of the matrix of $T$
with respect to a suitable basis.

 $\bullet$ $T$ is hyperbolic iff $\chi_{o T}(1) < 0$.

If $T$ is  parabolic then the minimal  polynomial $m_{ T}(x)$ must be divisible by $(x-1)^3$. On the other hand, an elliptic transformation is semisimple.

$\bullet$ $T$ is parabolic (resp. elliptic), iff $\chi_{o T}(1) > 0$, and $m_{ T}(x)$ is  divisible (resp. is not divisible),  by  $(x-1)^2$.

$\bullet$ Note that the above two statements provide a {\it finite algorithm} to test whether $T$ is hyperbolic, parabolic or elliptic.

\medskip Some other simple criteria for recognizing the type can be obtained by just looking at the trace. Note that for elliptic and parabolic transformations
all eigenvalues have absolute value  1. So

$ \bullet$ If $trace \,T > n + 1$, then $T$ is hyperbolic.

The eigenvalues of $T^u$ are $u$-th powers of the eigenvalues of $T$. Only the hyperbolic transformation has eigenvalue $r > 1$, and its multiplicity is $1$. So for large
$u$ the eigenvalue $r^u$ of $T^u$ will dominate the others. So

$ \bullet$ $T$ is hyperbolic iff for all sufficiently large $u$, $trace \, T^u > n + 1$.

\begin{proof}
Let $T$ be hyperbolic. Then it has a
pair of real eigenvalues $\{ \lambda, \lambda^{-1} \}$,
$\lambda>1$. All other eigenvalues have absolute value $1$.
Let the other eigenvalues be $ u_1,....,u_{n-1}$. So we have,
$$trace\;T= \lambda + \lambda^{-1}+ \sum_{i=1}^{n-1} u_i.$$
$$trace\;T^d=\lambda^d+\lambda^{-d}+\sum_{i=1}^{n-1} {u_i}^d.$$
As $\lambda>1$ and $|u_i|=1$ for $i=1,2,..,n-1$, we have for any
 $d \in {\mathbb N}$,
$$|trace\;T^d|=
|\lambda^d+{\lambda^{-d}}+u_1^d+...+ u_{n-1}^d|\geq \lambda ^d-(n-1).$$
Since $\lambda>1$, we can find an $u\in {\mathbb N}$
such that $\lambda^u > 2n$. i.e. $trace\;T^u>n+1$.
 On the other hand, if $T$ is parabolic or elliptic, then $|u_i|=1$ for
$i=1,2,...,n+1$. Hence for any $d$,
$$|trace\;T^d|=\big|\sum_{i=1}^{n+1} u_i^d\big| \leq n+1.$$
\end{proof}
For the cases $n = 2, 3$ we have  very neat criteria. Compare the usual criteria given in the $PSL(2, \R)$ model of $I_o(\H^2)$, or the $PSL(2, \C)$ model of $I_o(\H^3)$ cf. \cite{beardon}, \cite{ratcliffe}, or appendix A below. Note that the orientation-reversing case, either for $n =2$ or for $n=3$,  is usually not treated in the literarature.

$\bullet$ Let $n = 2$.

i) orientation-preserving case: $T$ is hyperbolic, resp. parabolic, resp. elliptic, iff
$trace \, T > 3$, resp.  $ trace \, T = 3$, resp. $ trace \, T < 3$.

ii) orientation-reversing case: $T$ is hyperbolic, resp.   elliptic, iff
$trace \, T > 1$, resp. $trace \, T  = 1$.

$\bullet$ Let $n = 3$.

Let $T$ be in $I(Q)$, $T \not= identity$.

i) orientation-preserving case: a)    $T$ is $1$-rotatory hyperbolic iff   $1$ is not an eigenvalue. $\;\;\;$ b) $T$ is $0$-rotatory hyperbolic, iff 1 is an eigenvalue and $trace \, T > 4$.     c) If $1$ is an eigenvalue then $T$ is elliptic (resp. parabolic) iff $trace \; T < 4$, (resp. $trace\; T = 4$.)

ii)  orientation-reversing case: $T$ is hyperbolic,   resp. elliptic, iff
$trace \, T > 2$, resp.  $ trace \, T \le  2$.

In these dimensions, by remarkable isomorphisms, the group $I(\H^2)$, resp. $I_o(\H^3)$ can be identified with the group $PGL(2,\R)$, resp. $PGL(2, \C)$. In the appendix A we give algebraic criteria using these isomorphisms. We have treated the orientation-reversing case also.

$\bullet$ In higher dimensions although it is hard to compute individual eigenvalues of $T$, the
$trace\; T$ and $trace\; T^u$ are easily computable by using
the well-known Newton's identities.

Let $\chi_T(x)=x^{n+1}-a_1 x^n + a_2 x^{n-1} +.....+ (-1)^{n+1} a_{n+1}$.

In our case, $a_{n+1}$ is $+1$ or $-1$ according as $T$ is orientation
preserving or orientation reversing.

Let $p_k$ be the sum of $k$-th powers of eigenvalues of $T$. Since the eigenvalues of $T^k$ are the $k$-th powers of eigenvalues of $T$, we have $p_k= trace \; T^k$.
The Newton's identities (cf. \cite{mcd}, \cite{jkv})
express $p_k$'s in terms of the coefficients
of the characteristic polynomial of $T$ as follows:

1) For $k \leq n+1$, $p_k=det \begin{pmatrix} a_1 & 1 & 0 & 0 &... & 0\\
2a_2 & a_1 & 1 & 0 & ... & 0 \\ 3a_3 & a_2 & a_1 & 1 & ... & 0 \\
... & ... & ...& ...& ...&...\\ ka_k & a_{k-1} & a_{k-2} & ...&... & a_1 \end{pmatrix}$.

2) For $k>n+1$, $p_k=a_1 p_{k-1} - a_2 p_{k-2} + ... + (-1)^{n-1} p_{k-n}$.

\medskip \noindent So we have,

$trace\; T=a_1$,

$trace\;T^2=p_2= a_1^2-2a_2$,

$trace \;T^3=p_3=a_1^3-3a_1a_2+3a_3$,

$trace\;T^4=p_4=a_1^4-4a_1^2a_2+4a_1a_3+2a_2^2-4a_4$, ....

\noindent If either of these numbers is $>n+1$, then $T$ is hyperbolic.

$\bullet$  Suppose $T$ is hyperbolic and has eigenvalue $\lambda$ such that
$\lambda \geq a >1$. Then for any real $r>0$, and  $m>\frac{ln\;(n-1+r)}{ln\;a}$ we have  $trace\; T^m>r$.

 \begin{proof}
We have for any positive integer $k$,
\begin{equation}\label{eq}
 trace \; T^k \geq \lambda^k-(n-1)> a^k-(n-1)
\end{equation}

\noindent
Since $a>1$, choose $m$ such that
$a^m-(n-1)>r$, or,  $m> \frac{ln\; (n-1+r)}{ln\;a}$.
\end{proof}

{\corollary
An isometry $T$ of $\H^n$ is hyperbolic if and only if
the sequence $\{trace\; T^k\}$ is divergent.}

\begin{appendix}
\section{Alternative criteria for $I(\H^2)$ and $I(\H^3)$, a Lie group theoretic and dynamic perspective}
In low dimensions there are two remarkable Lie theoretic isomorphisms. The group $SO_o(2, 1)$ is isomorphic to $PSL(2, \R)$ and $SO_o(3, 1)$ is isomorphic to $PGL(2, \C)$. These allow us to identify $I_o(\H^2)$ and $I_o(\H^3)$ with the groups $PSL(2, \R)$ and $PGL(2, \C)$ respectively.
\subsection{Isometries of $\H^3$}
Consider the disk model of the hyperbolic $3$-space, and identify its conformal boundary with the extended complex plane $\hat \C=\C \cup \{\infty\}$. The group $I(\H^3)$ can be identified with the group
$\m(2)$ of all M$\o$bius transformations of $\hat \C$. Recall that $\m(2)=\m^+(2) \cup \m^-(2)$, where
$$\m^+(2) = \{z \mapsto \frac{az+b}{cz+d}\;|\;a,b,c,d \in \C,\;ad-bc
  \neq 0 \},$$
$$\m^-(2) = \{z \mapsto \frac{a \bar z+b}{c \bar z+d}\;|\;a,b,c,d \in
  \C,\;ad-bc \neq 0 \}.$$
In particular, $I_o(\H^3)$ is isomorphic to $PGL(2, \C)$. Let $\sigma_0$ be the reflection in the real line in $\C$, i.e. $\sigma_0(z)=\bar z$.  The coset $PGL(2, \C)\sigma_0$ is isomorphic to the group of orientation-reversing isometries of $\H^3$.
Note that we consider $PGL(2, \C)$ rather than the usual $PSL(2, \C)$ in the literature. In general over fields $\F \neq \C$, the groups $PGL(2, \F)$ and $PSL(2, \F)$ are not always isomorphic. So we find it natural to consider the group $PGL(2, \C)$ and to avoid the ``det = 1" normalization for formulating the criteria.

Before proceeding further we introduce a new terminology of ``M$\o$bius co-ordinates". The coordinate $z$ in $\C$, extended to $\hat \C$ by setting $z(\infty) = \infty$, is called a {\it M$\o$bius coordinate} on $\hat \C$. It is {\it not} a complex coordinate in the sense the terminology  is used in the theory of manifolds. Any $PGL(2, \C)$-translate will also be termed
 as a M$\o$bius coordinate.

\subsection{Isometries of $\H^2$}
In the upper-half space model, $I(\H^2)$ is a subgroup of $\m(2)$ and is given by
the linear fractional transformations with real co-efficients, i.e.
$I(\H^2)=\m^+(1) \cup \m^-(1)$, where
$$\m^+(1) = \{z \mapsto \frac{az+b}{cz+d}\;|\;a,b,c,d \in \R,\;ad-bc > 0 \},$$
$$\m^-(1) = \{z \mapsto \frac{a \bar z+b}{c \bar z+d}\;|\;a,b,c,d \in
  \R,\;ad-bc < 0 \}.$$
Thus $I(\H^2)$ can be identified with the group $PGL(2, \R)$. Let $GL_+(2, \R)$, resp. $GL_-(2, \R)$, be the subgroup of all elements with positive, resp. negative,  determinant in $GL(2, \R)$. Then the group $I_o(\H^2)$ is isomorphic to $PGL_+(2, \R)$, which can be identified with $\m^+(1)$. The other component of the isometry group may be identified with $GL_-(2, \R)$.

In the disk model, the group $I_o(\H^2)$ consists of the usual holomorphic transformations of the unit disk $\D^2$. Recall that such a transformation is of the form
$$f_{a, \phi}: z \mapsto e^{i \phi} \frac{z-a}{1-\bar a z},\; 0 \leq \phi \leq 2 \pi,\; |a|<1.$$
The orientation-reversing isometries are given by
$$\bar f_{a, \phi}: z \mapsto  e^{i \phi}\frac{\bar z-a}{1-\bar a \bar z},\; 0 \leq \phi \leq 2 \pi,\; |a|<1.$$
Let $\R_+$ denote the group of all positive reals, and let
$$SU(1,1)=\bigg \{\begin{pmatrix} a & \bar c \\ c & \bar a\end{pmatrix}\;|\; |a|^2-|c|^2=1\bigg \}.$$
Then it follows from the above that $I_o(\H^2)$ lifts to the subgroup
$\R_+ \times SU(1,1)$ of $GL(2, \C)$.

Observe that the co-efficient matrix $M$ of $f_{a, \phi}$ has the following decomposition:
$$M=\begin{pmatrix}e^{i \frac{\phi}{2}} & 0 \\ 0 & e^{-i \frac{\phi}{2}}\end{pmatrix}
\begin{pmatrix}1 & -a \\ -\bar a & 1 \end{pmatrix}.$$
It turns out to be the polar decomposition, or the KP decomposition, of $f_{a, \phi}$.

\subsection{The criteria} First note that the quotient map $GL(2, \C) \to PGL(2, \C)$ is surjective, and it maps the conjugacy class of an element $A$ onto the conjugacy class of the corresponding element $\hat A$. Since the trace and determinant of $A$ are the conjugacy invariants in $GL(2, \C)$, we can take $\frac{(trace\;A)^2}{det\;A}$  as the conjugacy invariant for $\hat A$ in $PGL(2, \C)$.

Now suppose $A$ is an element in $GL(2, \C)$ and let $\hat A$ be the corresponding M$\o$bius transformation on $\hat \C$. Let $A=\begin{pmatrix} a & b \\ c & d \end{pmatrix}$. Then
$$\hat A(z)=\frac{az+b}{cz+d}.$$
By Jordan theory we know the conjugacy classes in $GL(2, \C)$.

{\it Case} (i). Suppose $A$ is conjugate to the diagonal matrix
$D_{\lambda, \mu}= \begin{pmatrix} \lambda & 0 \\0 & \mu \end{pmatrix}$, i.e. there exists $P$ in $GL(2, \C)$ such that
$D_{\lambda, \mu}=PAP^{-1}$. Considering $\hat P: z \mapsto w$ as a M$\o$bius change of co-ordinates on $\hat \C$, we have
$$\hat A(w)=  \frac{\lambda}{\mu} w.$$
Further note that $D_{\lambda, \mu}$ is conjugate to $D_{\mu,
\lambda}$. So the size $|\frac{\lambda}{\mu}|$ has no significance
for the conjugacy problem.

Now there are the following possibilities.

(a) $  \frac{\lambda}{\mu}$ is not a real number and $|  \frac{\lambda}{\mu}| \neq 1$. then $A$ acts as an $1$-rotatory hyperbolic. Classically these are known as \emph{loxodromic}.

(b) If $|  \frac{\lambda}{\mu}|=1$ and $\frac{\lambda}{\mu} \neq 1$, then $A$ acts as an $1$-rotatory elliptic.

If $  \frac{\lambda}{\mu} = -1$, then $A$ acts as an $1$-rotatory elliptic with rotation angle $\pi$, and we agree to call it a half-turn.

(c) $  \frac{\lambda}{\mu}$ is a real number. Suppose $|  \frac{\lambda}{\mu}|\neq 1$.

If $  \frac{\lambda}{\mu} >0$, then $A$ acts as a $0$-rotatory hyperbolic. We agree to call it a stretch.

If $  \frac{\lambda}{\mu} <0$, then $A$ acts as an $1$-rotatory hyperbolic with rotation angle $\pi$, and in this case we agree to call it a stretch half-turn.

 {\it Case }(ii). Suppose $A$ is conjugate to the upper-triangular matrix $T_{\lambda}=\begin{pmatrix}\lambda & 1 \\ 0 & \lambda \end{pmatrix}$. In this case $\hat A$ takes the form $w \mapsto w + \frac{1}{\lambda}$, and we see that $A$ acts as a $0$-rotatory parabolic. We agree to call it a translation.

\medskip In the orientation-reversing case, lifting $f: z \mapsto \frac{a \bar z + b}{c \bar z +d}$
to $A$ has no dynamic significance for $I(\H^3)$, for the
``correct"  lift is $A\sigma_0$. Hence the conjugacy classes in
$GL(2, \C)$ do not determine the conjugacy classes of the
orientation-reversing isometries. However one dynamic invariant of
$f$ is $f^2$, which corresponds to $A \sigma_0 A \sigma_0$, and is
an orientation-preserving transformation. The associated matrix to
$f^2$ is given by $B=A \bar A$. As we shall see in the next
theorem, when $B \neq \lambda I$, $\lambda>0$,  $B$ determines
$A$. The elusive case is when $B=\lambda I$. It will turn out that
in this case there are exactly two conjugacy classes.  Either $f$
is a $0$-rotatory elliptic inversion, and the fixed point set of
$f$ is a circle, or $f$ is an $1$-rotatory elliptic inversion, and
$f$ is an {\it antipodal} map without fixed points. An
\emph{antipodal map}, by definition, is an orientation-reversing
map of order $2$ and without a fixed point on $\hat \C$. It turns
out that these maps are conjugate to $\mathfrak a: z \mapsto
-\frac{1}{\bar z}$.  In the disk model of $\H^3$, these maps fix a
unique point in the disk, and interchange the end points of a
geodesic passing throgh the fixed-point.

\subsubsection{Criteria for $I(\H^3)$}
\begin{theorem}\label{ach3}  Let $f$ be an isometry of $\H^3$ induced
by  $A \in GL(2, \C)$. Define
$$c(A)=\frac{(trace \;A)^2}{det \;A}.$$

I.  Suppose $f$ is orientation-preserving.

 $(i)$  If $c(A)$ is in $\C-\R$, then $f$ is an
$1$-rotatory hyperbolic, and is different from a stretch half-turn.

 $(ii)$ If $c(A)$ is in $\R$, $|c(A)-2|>2$, then
$f$ is a stretch or a stretch half-turn  according as $c(A) >0$ or
$c(A) <0$.

 $(iii)$ If $c(A)$ is in $\R-\{0\}$, $|c(A)-2| <2$, then
$f$ is an $1$-rotatory elliptic.

$(iv)$ If $c(A)=0$, then $f$ is a half-turn.

 $(v)$ If $c(A)$ is in $\R-\{0\}$, $|c(A)-2|=2$,
$A \neq \lambda I$, then $f$ is a translation.

$(vi)$  If $c(A)$ is in $\R-\{0\}$, $|c(A)-2|=2$, $A=\lambda I$,
then $f$ is the identity.

\medskip II. Suppose $f$ is orientation-reversing.  Let
$B=A \bar A$. Then $c(B)$ is a real positive number.

 $(i)$ If $|c(B)-2|<2$, then $f$ is an $1$-rotatory elliptic inversion.

 $(i)$ If $|c(B)-2|>2$, then $f$ is a $0$-rotatory hyperbolic inversion.

$(iii)$ If $|c(B)-2|=2$ and $B \neq \lambda I$, then $f$ is a $0$-rotatory parabolic
inversion.

$(iv)$ If $|c(B)-2|=2$ and $B=\lambda I$, then $f$ acts as an
inversion in a circle, or an antipodal map. In both cases $f$ is an elliptic inversion. We detect these cases as follows.

If there exists a non-zero $u$ in $\C$, unique up to a non-zero
real multiple, such that the matrix $uA$ is of the form:
$\begin{pmatrix}a & b \\
c & -\bar a \end{pmatrix}$,
 where $b$, $c$ are reals, then $f$ acts as an
inversion in a circle, resp. antipodal map according as  $\det uA<0$, resp.  $\det uA>0$.

The inversion in a circle is a $0$-rotatory elliptic-inversion and the antipodal map is a $1$-rotatory elliptic inversion of $\H^3$.
\end{theorem}

\begin{proof}
$I$.   Let the matrix $A$ induce an orientation-preserving isometry.  Let $\lambda$ and $\mu$ be eigenvalues of $A$. So,
$$c(A)=\frac{(trace\;A)^2}{det\;A}=\frac{(\lambda+\mu)^2}{\lambda
  \mu}={\frac{\lambda}{\mu}}+{\frac{\mu}{\lambda}}+2.$$
 Let $\frac{\lambda}{\mu}=re^{i \theta}$, i.e. $c(A)=re^{i \theta}+{\frac{1}{r}}e^{-i \theta}+2$.

(1) If $c(A)$ is  non-real, then the imaginary part of
 $re^{i \theta}+{\frac{1}{r}}e^{-i \theta}$ must be non-zero.  This
is possible if and only if $r \neq 1$ and $\theta \neq 0, \pi$. Thus $A$ acts as a $1$-rotatory hyperbolic and  is different from a stretch half-turn.

(2) Now suppose $c(A)$ is a real number. Then either $\sin \theta=0$, or $r=1$. So $\theta=0$, or $\pi$, or, $r=1$.

 If $|c(A)-2|<2$, then we must have $r=1$, $\theta \neq 0,\pi$. Hence $A$ acts as an $1$-rotatory elliptic, different from a half-turn.

 If $|c(A)-2|>2$, then we must
have $\theta=0\;\hbox{or}\;\pi$, $r \neq 1$. Thus $A$
acts as a stretch or a stretch half-turn according as  $\theta =0$ or $\theta=\pi$.

 If $|c(A)-2|=2$, then we must have
$r=1$, $\theta=0$ or $\pi$. If $\theta=\pi$, we have $\frac{\lambda}{\mu}=-1$, i.e. $A$ acts as a half-turn. If $\theta=0$ we have
$\lambda=\mu$.
 If $A \neq \lambda I$, then $A$ is
upper triangulable, hence $A$ acts as a translation.
 If $A=\lambda I$, it induces the identity map.

\smallskip $II$.  Let $f$ be an orientation-reversing isometry and the corresponding M$\o$bius transformation, again denoted by $f$, is given by
$$f: z \mapsto \frac{a\bar z+ b}{c \bar z+d}.$$
Then $A\sigma_0A\sigma_0=A \bar A$ induces the
  orientation-preserving M$\ddot{\hbox{o}}$bius transformation $f^2$ on $\hat
  \C$, where
$$\sigma_0 A \sigma_0=\bar A=\begin{pmatrix} \bar a & \bar b \\
\bar c & \bar d \end{pmatrix}.$$
Let $B=A \bar A$. Then by the previous case we have

(3) $|c(B)-2|<2 \Rightarrow$ $f^2$ is a
$1$-rotatory elliptic $\Rightarrow$ $f$ is a $1$-rotatory elliptic inversion.

(4) $|c(B)-2|>2 \Rightarrow f^2$ is a
stretch $\Rightarrow$ $f$ is a $0$-rotatory hyperbolic inversion.

(5) $|c(B)-2|=2,\;B \neq \lambda
I$. Then $f^2$ is a translation, and hence
$f$ is a $0$-rotatory parabolic inversion.

(6) Finally, consider the case when $|c(B)-2|=2,\;B = \lambda I$.
Note that
$$trace\;B=|a|^2+|d|^2+(b \bar c+\bar b c)=2 \lambda,\;
\hbox{det} \;B=|ad-bc|^2.$$
Thus $c(B)$ and $\lambda$ are real numbers. Since $\lambda=|a|^2+b \bar
c=|d|^2+\bar b c$, we must have $b \bar c \in \R$.

Let  $b \bar c \neq 0$. Then $b$ must be a real multiple of $c$. Let
 $b=uc=ure^{i \theta},\;u \in \R$.
 Multiplying $A$ by $e^{-i \theta}$, without loss of
generality, we may assume $A=\begin{pmatrix}a & b \\
c & d \end{pmatrix},\;b,c \in \R$.
The equation $B=A \bar A=\lambda I$ also gives us,  $c \bar a+d \bar
c=0$. Since $c$ is a real, we have $\bar a+d=0$, i.e. $\bar a=-d$.
Hence,  $A=\begin{pmatrix}a & b\\
c & -\bar a \end{pmatrix}$.
 The induced M$\ddot{\hbox{o}}$bius  transformation by
$A\sigma_0$ is
$$f: z \mapsto \frac{a \bar z+b}{c \bar z-\bar a}.$$
Suppose $c \neq 0$.  Note that $f$ has a fixed point on $\C$ if and only if $f(z)=z$ has a solution. We see that
\begin{equation}
\frac{a \bar
  z+b}{c \bar z-\bar a}=z \Leftrightarrow (z-{\frac{a}{c}})
({\bar z}-{{\frac{\bar a }{\bar c}}})=\frac{a \bar a+bc}{{\bar c}^2}
=\frac{-det\;A}{{\bar c}^2}.
\end{equation}
Thus if $det \;A<0$, then the above equation has a solution and
$f$ has a  fixed point. The fixed point set is a circle $C$ with
radius $\sqrt{\frac{-det\;A}{{\bar c}^2}}$. This implies, $f$ must
be an inversion in $C$. If $det\;A>0$, then $f$ has no fixed-point
on $\hat \C$ and $f^2=id$. By definition $f$ will be the antipodal
map of $\hat \C$. An inversion in a circle  is a $0$-rotatory
elliptic inversion of $\H^3$. It follows from the representation
of the isometries in $O(3, 1)$ that the antipodal map acts as an
$1$-rotatory elliptic inversion of $\H^3$, with rotation angle
$\pi$.

If $c=0$, then $A\sigma_0$ acts as
$f: z \mapsto a \bar zd^{-1}+bd^{-1}$. Since $f^2$ is the
identity, it follows that $f$ is an inversion in a circle of $\hat \C$, and acts as an $0$-rotatory elliptic inversion of $\H^3$.
\end{proof}
\subsubsection{Criteria for $I(\H^2)$}
\begin{corollary}
Let $f$ be an isometry of $\H^2$ induced by $A$ in $GL(2, \R)$, resp. $\R_+ \times SU(1,1)$. Let $c(A)=\frac{(trace\;A)^2}{det A}$. Then $c(A)$ is a real number.

I. Suppose $f$ is orientation-preserving.

(i) If $|c(A)-2|<2$, then $f$ is a $1$-rotatory elliptic.

(ii) If $|c(A)-2|>2$, then $f$ is a stretch.

(iii) If $|c(A)-2|=2$ and $A \neq r I$, then $f$ is a translation.

(iv) If $|c(A)-2|=2$ and $A=rI$, then $f$ is the identity.

II. Suppose $f$ is orientation-reversing. Let $f$ be induced by $A$ in $GL(2, \R)$,     resp. $\R_+ \times SU(1,1)$. Let $B=A \bar A$.

(i) If $|c(B)-2|>2$, then $f$ is a $0$-rotatory hyperbolic inversion.

(i) If $B \neq rI$ and $|c(B)-2|<2$,  then $f$ is an $1$-rotatory elliptic inversion.

(iii) If $B=rI$, then $f$ is a $0$-rotatory elliptic inversion. In this case, $f$ is an inversion in a circle of $\hat \C$.
\end{corollary}
\begin{proof}
It is easy to see that every orientation-reversing isometry of $\H^2$ has exactly two fixed points on the boundary circle. Now the corollary follows from the above theorem.
\end{proof}

\section{$z$-Classes in the group $\H^n$}
We have seen the use of ``KP-decomposition" above. Similarly there is a ``KAN-" or the ``Iwasawa decomposition" for every reductive Lie group, in particular for $I_o(\H^n)$.
 The action of $AN$ on $\H^n$ is simply transitive. So modulo a choice of a base-point, which in the upper half space model we  take to be $(0, 0, ..., 1)$,   $\H^n$ can be considered as a {\it group}. As a curiosity, we visualize the stratification of $AN$ into $z$-classes by identifying $AN$ with  $\H^n$.

Let $\x$ denote a point in $\R^{n-1}$, which we take as a part of the ideal boundary (except for $\infty$ ) of $\H^n$. The points of $\H^n$ are
$(\x, x_n)$, $x_n > 0$. Then $A = \{f_r \;|\; r > 0\}$ acts by
$$f_r: (\x, x_n) \mapsto  (r\x, rx_n), \;r > 0,$$
and $N = \{g_{\bf a} \;|\; {\bf a} \in \R^{n-1}\}$ acts by
$$g_{\bf a}: (\x, x_n) \mapsto (\x + {\bf a}, x_n).$$
 The subgroup $N$ is normal in $AN$, and $AN$ is in fact a semi-direct product of $N$ by $A$, where $A \approx \R_{>0}$ acts on $N \approx \R^{n-1}$ by
 $$  \x \mapsto r\x.$$
  A general element of $AN$ is $g_{\bf a}\circ f_r$, which acts on $\H^n$ by
  $$(\x, x_n) \mapsto  (r\x + {\bf a}, rx_n). $$
By the identification of $AN$ with $\H^n$ as indicated above, we
see that the element $g_{\bf a}\circ f_r$ corresponds to the point
$({\bf a}, r)$.

\begin{theorem}
1. The non-identity conjugacy classes in $\H^n$ are represented by  $f_r$, $r \neq 1$, and  $g_{\bf a}$, ${\bf a} \neq 0$.

i) The conjugacy class of $f_r$, $r \neq 1$, in the above identification, corresponds to the hyperplane $x_n = r$.

ii) The conjugacy class of $g_{\bf a}$, ${\bf a} \neq 0$ corresponds to the ray $(r{\bf a}, 1)$, $r > 0$, in the hyperplane $x_n = 1$.

2. There are only two non-identity z-classes, which can be
visualized as

(i) the complement of the hyperplane $x_n=1$ in $\H^n$, and

(ii) the hyperplane $x_n=1$, with puncture at $({\bf 0}, 1)$.
\end{theorem}
\begin{proof}
Let $f=g_{\bf a} \circ f_r$, $r \neq 1$. Let $\x_0=-(r-1)^{-1}{\bf a}$. Then
$$g_{\x_0} f g^{-1}_{\x_0}=f_r.$$
It is easy to see that $f_r$ is not conjugate to $f_s$, $r \neq s$.

Now suppose $r=1$. Then $f=g_{\bf a}$. Now for any element $h=g_{\bf b} \circ f_s$, we have
$$hfh^{-1}=g_{s \bf a}.$$
Thus $g_{\bf a}$ is conjugate to $g_{s \bf a}$, $s>0$, and if ${\bf b} \neq t{\bf a}$ for some $t>0$, then $g_{\bf b}$ is not conjugate to $g_{\bf a}$.

Hence the conjugacy classes in $AN$ are represented by $f_r$, for $r>0$, and $g_{\bf a}$, where $\bf a$ is unique up to a multiplication by $r>0$. In the above identification of $AN$ with $\H^n$, the conjugacy class of $f_r$ corresponds to
$$\{({\bf a}, r)\;|\;{\bf a} \in \R^{n-1}\},$$
and the conjugacy class of $g_{\bf a}$ corresponds to
$$\{(\x, 1)\;|\;\x=r {\bf a} \hbox{ for } r>0\}.$$

For ${\bf a} \neq 0$,  the group $Z(g_{\bf a}) = N$ is normal. For $r, s \neq 1$, $Z(f_r) = Z(f_s) = A$. Since for $r \neq 1$, $g_{\bf a}\circ f_r$ is conjugate to $f_r$, it follows that for all $r \neq 1$, all $Z(f_r)$, and $Z(g_{\bf a}\circ f_r)$  form a single conjugacy class of subgroups. Correspondingly there are exactly two non-identity $z$-classes in $\H^n$. The $z$-class corresponding to $f_r$ is given by
$$\bigcup\{x_n=s\;|\; s \neq 1\}.$$
The $z$-class of $g_{\bf a}$ is given by
$$\{(\x, 1)\}|\; \x \neq {\bf 0}\}.$$
\end{proof}


\end{appendix}

\end{document}